\documentclass{article}
\usepackage{amsmath, amsthm, amssymb, braket, graphicx, color}
\usepackage[all]{xy}
\usepackage[colorlinks = true]{hyperref}
\usepackage[normalem]{ulem}
\usepackage[left = 4cm, right = 4cm]{geometry}

\newtheorem{definition}{Definition}
\newtheorem{Example}{Example}
\newtheorem{Lemma}{Lemma}
\newtheorem{Theorem}{Theorem}
\newtheorem{Corollary}{Corollary}

\DeclareMathOperator{\tr}{trace}
\DeclareMathOperator{\vc}{vec}

\title{Seidel switching for weighted multi-digraphs and its quantum perspective}
\author{Supriyo Dutta, \\ Department of Mathematics,\\ Indian Institute of Technology Jodhpur.\\ Email: \texttt{dutta.1@iitj.ac.in}\\
\\
Bibhas Adhikari, \\ Department of Mathematics, \\ Indian Institute of Technology Kharagpur. \\ Email: \texttt{bibhas@iitkgp.erent.in}\\
\\
Subhashish Banerjee, \\ Department of Physics, \\ Indian Institute of Technology Jodhpur. \\ Email: \texttt{subhashish@iitj.ac.in}}
\date{}

\begin{document}
\maketitle

\begin{abstract}
Construction of graphs with equal eigenvalues (co-spectral graphs) is an interesting problem in spectral graph theory. Seidel switching is a well-known method for generating co-spectral graphs. From a matrix theoretic point of view, Seidel switching is a combined action of a number of unitary operations on graphs. Recent works \cite{braunstein2006laplacian} and \cite{adhikari2012} have shown significant connections between graph and quantum information theories. Corresponding to Laplacian matrices of any graph there are quantum states useful in quantum computing. From this point of view, graph theoretical problems are meaningful in the context of quantum information. This work describes Seidel switching from a quantum perspective. Here, we generalize Seidel switching to weighted directed graphs. We use it to construct graphs with equal Laplacian and signless Laplacian spectra and consider density matrices corresponding to them. Hence Seidel switching is a technique to generate cospectral density matrices. Next, we show that all the unitary operators used in Seidel switching are global unitary operators. Global unitary operators can be used to generate entanglement, a benchmark phenomena in quantum information processing.
\end{abstract}

\newpage
\section{Introduction}

Graph theory \cite{west2001introduction, bapat2010graphs} is a well developed branch of mathematics with applications to different branches of science and humanities. Graphs are used to analyse structure of a complex system particularly in social, economical and biological networking as well as in computer architecture.  Graph theory has played an important role in the development of information theory\cite{lovasz1979shannon}. Quantum information theory \cite{nielsen2010quantum} has been one of the promising scientific developments of recent times and draws usefully from both physics and mathematics \cite{ braunstein2006laplacian, adhikari2012, berkolaiko2013introduction, cabello2014}.

A graph $G = (V(G), E(G))$ is a pair of vertex set $V(G)$ and edge set $E(G)$. The number of elements in $V(G)$ is called order of the graph. The edge set, $E(G) \subset V(G) \times V(G)$. A loop is an edge of the form, $(v,v) \in E(G)$ at the vertex $v$. A multi-digraph has multiple directed edges $(u,v), (v,u) \in E(G)$. A weighted graph consists of a weight function $w: E(G) \rightarrow \mathbb{R}, w(u,v) = w_{u,v}$. In general, a graph $G$, in this article, is a weighted digraph with multiple edges and loops, except specifically mentioned. Such a graph is represented by an adjacency matrix of $G$ denoted by $A(G)$ and defined by,
$$(A(G))_{ij} = a_{ij}(G) = 
\begin{cases}
0 & \text{if}~ (i, j) \notin E(G),\\
w(i, j) & \text{if}~ (i, j) \in E(G),\\
w(i, i) & \text{if}~ (i, i) \in E(G).\\
\end{cases}$$
We assume that $w(i,i) > 0$, when $(i,i) \in E(G)$. When no confusion arises we right $a_{ij}(G) = a_{ij}$. Degree matrix of $G$ is $D(G)=\mbox{diag}\{d_1, d_2, \hdots, d_k\}, d_i = \sum_{j = 1}^k|a_{ij}|$. The Laplacian \cite{merris1994laplacian} and the signless Laplacian \cite{cvetkovic2009towards1} of $G$ are $L(G)=D(G)-A(G)$ and $ Q(G)=D(G)+A(G)$, respectively, when $w(i,j)=w(j,i)$ for all $i,j\in V(G)$. A simple graph is a special case of a weighted multi-digraph. It does not have loop, multiple directed edges and each edge has weight $1$. Two weighted multi-digraphs $G$ and $H$ are said to be isomorphic if there is a bijective function $f: V(G) \rightarrow V(H)$, such that, $(u,v) \in E(G)$ if and only if $(f(u), f(v)) \in E(H)$ and $w(u,v) = w(f(u),f(v))$. For two isomorphic graphs $G$ and $H$ there is a permutation matrix $P$, such that, $A(H) = P^tA(G)P, L(H) = P^tL(G)P$ and $Q(H) = P^tQ(G)P$. Spectra of a matrix $X$ is the multi-set containing all the eigenvalues of $X$, denoted by $\Lambda(X)$. Spectra of a graph is $\Lambda(A(G))$. In a similar fashion, $\Lambda(L(G))$ and $\Lambda(Q(G))$ are Laplacian and signless Laplacian spectra, respectively. Two graphs $G$ and $H$ are co-spectral, L-co-spectral, and Q-co-spectral, if $\Lambda(A(G)) = \Lambda(A(H))$, $\Lambda(L(G)) = \Lambda(L(H))$ and $\Lambda(Q(G)) = \Lambda(Q(H))$, respectively. Graph isomorphism problem is an NP-Hard problem in general. Being co-spectral is a necessary condition for being isomorphic. Hence, finding non-isomorphic cospectral graphs is an interesting problem.

In quantum mechanics, density matrix $\rho$, represents a quantum state, which is normalised, $\tr(\rho) = 1$, positive semi-definite Hermitian matrix. In general, a density matrix can be written as \cite{horodecki2009quantum} $\rho = \sum_ip_i\ket{\phi_i}\bra{\phi_i}, 0 \le p_i \le 1; \sum_ip_i = 1$. Here, $\ket{\phi_i}$ is a column vector, called state vector, belonging to a Hilbert space $\mathcal{H}$ over $\mathbb{C}$. We denote the conjugate transpose of $\ket{\phi}$ by $\bra{\phi}$. The usual matrix product $\ket{\phi}\bra{\phi}$ is called outer product. A state vector of dimension two is called qubit. A quantum state may be distributed between a number of parties. Each of these parties are equipped with different Hilbert spaces, say, $\mathcal{H}_1, \mathcal{H}_2, \dots \mathcal{H}_n$. Hence, the complete Hilbert space can be collectively described as a tensor product of Hilbert spaces of individual parties, for example, $\mathcal{H}^{(\otimes n)} = \mathcal{H}_1 \otimes \mathcal{H}_2 \otimes \dots \otimes \mathcal{H}_n$. Throughout this article, $\otimes$ denotes the usual tensor product. A state is called multi-partite if the state vectors in the expression of $\rho$ belong to $\mathcal{H}^{(\otimes n)}$.

Every graph represents quantum states. The Laplacian, $L(G)$, and the signless Laplacian, $Q(G)$, are positive semi-definite matrices associated to a graph $G$. They are Hermitian matrix provided $w(u,v) = w_{u,v} = w(v,u)$, and $(v,u) \in E(G)$ whenever $(u,v) \in E(G)$. The density matrices associated with $L(G)$ and $Q(G)$ are defined by 
$\rho_l(G) = \frac{L(G)}{\tr(L(G))} ~\cite{braunstein2006laplacian}~ \mbox{and}~ \rho_q(G) = \frac{Q(G)}{\tr(Q(G))} ~\cite{adhikari2012},$
respectively. We collectively denote $\rho_l(G)$ and $\rho_q(G)$ by $\rho(G)$. We have discussed earlier \cite{dutta2016bipartite} that some important quantum mechanical properties of $\rho(G)$, such as entanglement, are not invariant under graph isomorphism.

Geometrically a quantum state $\ket{\psi}$ is depicted by a vector, the Bloch vector, in a sphere called the Bloch sphere, when $\ket{\psi}$ belongs to a Hilbert space of dimension $2$. In higher dimensions, the generalization of the idea of the Bloch sphere becomes quite intricate. The spectra of the density matrix of a state and its Bloch vector representation has been studied in detail in \cite{ozolsgeneralized}. Laplacian and signless Laplacian co-spectral graphs represent quantum states having the Bloch vectors of equal length. Thus, constructing L-co-spectral and Q-co-spectral graphs is tantamount to generating quantum states with Bloch vectors of equal length. When $L(G)$ ($Q(G)$ respectively) is unitary equivalent to $L(H)$ ($Q(H)$), then the quantum states $\rho_l(G)$ ($\rho_q(G)$) and $\rho_l(H)$ ($\rho_q(H)$) have the same spectra.

In graph theory, Seidel switching is a well known method for constructing co-spectral graphs. For a simple graph $G$, this is a unitary operation on $A(G)$ to generate a cospectral graph $H=(V(H),E(H))$. $V(H)=V(G)$ and some edges of $G$ are removed and new edges are introduced. This switching, introduced by Seidel \cite{Seidel1974} is given by,
$$ E(H) = \{xy \in E(G)| x, y \in S \hspace{.1 cm} \text{or} \hspace{.1 cm} x,y \notin S\} \cup \{xy \notin E(G) | x \in S \hspace{.1 cm} \text{and}\hspace{.1 cm} y \notin S\}, $$ where $S \subset V(G).$ Then $G$ and $H$ are called switching equivalent. Some recent works in this direction are \cite{butler2010note}. Graph isomorphism and construction of non-isomorphic cospectral graphs have been used in \cite{emms2009coined}, in the context of quantum computation .

In \cite{adhikari2012}, quantum states related to weighted graphs were introduced. Spectra of weighted graphs has been studied in literature and applied in network theory \cite{halbeisen2000reconstruction, milanese2010approximating} In this work, we generalize the concept of Seidel switching to weighted multi-digraphs. This provides an opportunity to study the combinatorial structures of co-spectral density matrices. In this paper, we will exhibit quantum mechanical properties of Seidel switching operation from a quantum mechanical perspective. In section $2$, we generalize it to the generation of L-co-spectral and Q-co-spectral weighted multi-digraphs. These play an important role in the graphical representation of quantum states \cite{adhikari2012}. In section 3, a number of interesting quantum mechanical properties of  Seidel switching are studied. As an interesting observation we show that the CNOT gate is a special type of Seidel operator. A number of notations used in section 3 are clarified at the appropriate juncture. We finally make our conclusions.

\section{A generalization of Seidel switching for weighted multidigraphs}\label{Seidel}

In this section, we generalize the notion of Seidel switching to generate L-co-spectral and Q-co-spectral graphs. The Laplacian and the signless Laplacian matrix, of a graph, can be used to construct another graph whose adjacency matrix consists of appropriately weighted loops. Seidel switching can be applied on this new adjacency matrix. Following the reverse procedure we can see that new adjacency matrix is the Laplacian or signless Laplacian of some other graph. In this way, we generate L-co-spectral and Q-co-spectral graphs using Seidel switching. This procedure is applicable for a particular class of graphs, defined below.

\subsection{Construction for cospectral graphs by Seidel switching}

We apply Seidel switching on a particular class of graphs, called Seidel graph. A regular graph has equal degree for all vertices. A graph $C$ is a subgraph of a graph $G$, if $V(C) \subset V(G)$ and $E(C) \subset E(G)$. When an edge $(u,v) \in E(G)$ and $u, v \in V(C)$ indicate that $(u,v) \in E(C)$, then the subgraph $C$ is called an induced subgraph of $G$.
\begin{definition}
A Graph is said to be a Seidel Graph if it satisfies the following conditions.
\begin{enumerate}
\item
The vertex set can be partitioned as $C_1, C_2, \dots , C_n , D$. $C_i$ contains $n_i\geq 2$ number of vertices. $V = \cup_{i = 1}^n C_i \cup D,$ $C_i \cap C_j = \phi$ $\forall i\neq j$ and $C_i \cap D = \phi$ $\forall i$.
\item
Subgraphs of $G$ induced by vertex sets $C_i$ and $D$ are regular.
\item
Any vertex $v \in D$ shall be adjacent to either $0$ or $n_i$ or $\frac{n_i}{2}$ number of vertices of $C_i$, for all $i$. When $v$ is adjacent to $\frac{n_i}{2}$ number of vertices of $C_i$ then the edges must have equal weights. Weights of the edges from the vertices of $C_i$ to $v$ are also equal.
\item When two vertices are linked with two directed edges, the edges are oppositely oriented.
\end{enumerate}
\end{definition}

Observe that two vertices in a Seidel graph are linked with at most two directed and weighted edges, the weights are real numbers, and the weights may be distinct. One vertex can have at most one loop. The weight of an edge between two vertices is zero if and only if there is no edge connecting them.

For any $i$, the vertices in $D$ can be classified into three categories.
\begin{itemize}
\item
{\bf Category 1 nodes:} Vertices which are adjacent to $n_i$ number of vertices of $C_i$. Let there be $p$ such vertices.
\item
{\bf Category 2 nodes:} Vertices which are adjacent ot $\frac{n_i}{2}$ number of vertices in $C_i$. Let there be $q$ such vertices.
\item
{\bf Category 3 nodes:} Vertices which are not connected to any vertices in $C_i$. Let there be $r$ such vertices.
\end{itemize}
Obviously, $ p + q + r = |D|$, the number of vertices in $D$.

The adjacency matrix of a Seidel graph is of the form
\begin{equation} \label{adjacency}
A = \begin{bmatrix}
C_1 & C_{12} & \dots &C_{1n} & D_1\\
C_{21} & C_2 & \dots &C_{2n} & D_2\\
\vdots & \vdots & \ddots & \vdots & \vdots\\
C_{n1} & C_{n2} & \dots &C_n & D_n\\
D^{(1)} & D^{(2)} & \dots & D^{(n)} & D
\end{bmatrix}_{|V(G)| \times |V(G)|}.
\end{equation}
In general $C_{ij} \neq C_{ji}^T$ and $D_i \neq D^{(i)T}$.

\begin{Example}
Familiar simple graphs like cycle, path, complete graphs and Petersen graphs are all Seidel graphs.
\begin{figure}
$$\xymatrix{1 \ar@(ul,dl) \ar@{-}[r]^2 \ar@/_/[d] \ar@/_2pc/[dd]
         &2 \ar@{-}[r]^3 \ar@/^/[ld]
         &3 \ar@/_2pc/[ll]_2 \ar@/^/[ddr] \ar@/^/[lld]
               \\ 4 \ar@(ul,dl) \ar@/_/[u] \ar@/^/[ru] \ar@/^/[urr]
                 &5 \ar@(ur,dr) \ar@{-}[l] \ar@{->}[rd] \ar@{-}[rrd]
                     \\ 6  \ar@/_2pc/[uu] \ar@/_/[r]
                        &7 \ar@/_/[l] \ar@/_/[r]
                        &8 \ar@/_/[l] \ar@/_/[r]
                        &9 \ar@/_/[l] \ar@/^/[uul]}$$
\caption{Seidel Graph}
\label{fig: Seidel}
\end{figure}
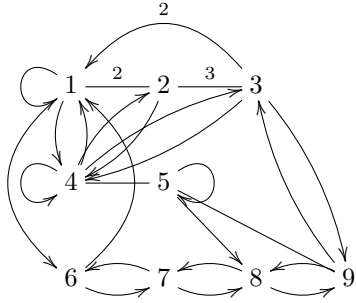
In figure~\ref{fig: Seidel}, the graph has multiple directed edges, some of which are weighted, indicated by the number over the edge connecting two vertices. Also, some vertices have loops. Consider $C_1 = \{1,2,3\}, C_2 = \{7,8,9,10\}$ and $D = \{4,5\}$. The subgraph induced by $C_1$ is a weighted regular graph of degree $5$. Vertex $4 \in D$ is adjacent to all the vertices of $C_1$ but no vertex of $C_2$. Also vertex $5 \in D$ is adjacent to no vertex in $C_1$ but two vertices in $C_2$. Also there are multiple directed edges between $C_1$ and $C_2$.
\end{Example}

The $m \times n$ matrix with all ones is denoted by $J_{m\times n}$. $j_n$ is a column vector of order $n$ with all ones. Observe that $U_n = \frac{2}{n}J_n - I_n$ is a unitary matrix. In fact, $U_n$ is unitary symmetric matrix, that is, $U_n^2 = I$ \cite{godsil1982}.

\begin{Lemma}\label{lemma1}
Let $A_{m \times n}$ be a matrix with constant row sum $r$. Then $$U_m A U_n = \frac{2r}{m}J_{m \times n} - \frac{2r}{n}J_{m\times n} + A.$$ In particular, $U_mAU_n=A$ when $m=n.$
\end{Lemma}
\begin{proof}
\begin{equation*}
\begin{split}
U_m A U_n &= (\frac{2}{m}J_m - I_n)A_{m\times n}(\frac{2}{n}J_n - I_n)\\
&= \frac{4}{mn}J_m A_{m\times n}J_n - \frac{2}{n}A_{m\times n}J_n - \frac{2}{n}J_mA_{m\times n} + I_m A_{m \times n}I_n\\
&= \frac{4rn}{mn}J_{m \times n} - \frac{2r}{n}J_{m \times n} - \frac{2r}{m}J_{m \times n} + A_{m \times n}\\
&= \frac{2r}{m}J_{m \times n} - \frac{2r}{n}J_{m\times n} + A\\
&= A \ \text{when,} \ m = n
\end{split}
\end{equation*}
\end{proof}

\begin{Lemma}\label{lemma2}
Let $x$ be a column vector with $2m$ entries, $m$ of which are zeros and remaining are ones, then $$U_{2m} x = j_{2m} - x$$
\end{Lemma}
\begin{proof}
$$U_{2m}x = (\frac{2}{2m}J_{2m} - I_{2m})x
= \frac{1}{m}J_{2m}x - x
= \frac{m}{m}j_{2m} - x
= j_{2m} - x$$
\end{proof}
If $x$ is a column vector with $2m$ number of components, $m$ of them being zero and other $m$ are equal constants, say $r$. Then
$$U_{2m}x = (\frac{2}{2m}J_{2m} - I_{2m})x
= \frac{1}{m}J_{2m}x - x
= \frac{rm}{m}j_m - x
= rj_m - x$$
If the sum of the elements of $x$ is $s$, then $U_{m}x = \frac{2s}{m}j_{m} - x$.

\vspace{.5cm}
{\bf Procedure for switching of Seidel graph}
\begin{enumerate}
\item \label{cons3}
Let $v \in D$ is adjacent to $\frac{n_i}{2}$ number of vertices in $C_i$. Do same for edges from the vertices of $C_i$ to $v$.
\item \label{cons4}
Let $v \in D$ is adjacent to all of vertices of $C_i$. Arrange new weights to all those edges as follows. New weights vector $w = \frac{2s}{n_i}j_{n_i} - x$, where $s = \sum_{t = 1}^{n_i}x_t$, $x$ is the vector containing edge weights.
\item \label{cons5}
New edges between the vertices of $C_i$ and $C_j$ and their weights are given by $U_mAU_n = \frac{2r}{m}J_{m\times n} - \frac{2r}{n}J_{m\times n} + A$, where, $A$ is the adjacency sub-matrix representing the edges between the vertices of $C_i$ and $C_j$.
\end{enumerate}
After all these changes, a new graph is formed. Let it be denoted by $G^{\pi}$.

\begin{Theorem}
Let G be a Seidel graph, then $G^{\pi}$ and $G$ are cospectral.
\end{Theorem}
\begin{proof}
The adjacency matrix of a Seidel graph has the form given in equation \ref{adjacency}. Let $C_{ij}$ be the submatrix of $A(G)$ that represents the adjacency relations between the vertices of $C_i$ and $C_j$ which are subgraphs of $G$ as define above. The following statements can be proved by Lemmas \ref{lemma1} and \ref{lemma2}.
\begin{enumerate}
\item
$U_{2n} x = r j_m - x$, where $r$ is the edge weight and $m = 2n$ an even number. This will give us corresponding changes when we remove $\frac{n_i}{2}$ number of vertices and add the vertex $v$ with another $\frac{n_i}{2}$ vertices. This explains the point \ref{cons3} of the construction given above.
\item
$U_n x = \frac{2s}{n} j_n - x$. This provides explanation for point \ref{cons4} of the construction.
\item
$U_m C_{ij} U_n = \frac{2r}{m} J_{m \times n} - \frac{2r}{n} J_{m \times n} + C_{ij}$ gives all the edges with two vertices, one in $C_i$ and another in $C_j$. Hence, this formula will give all the changes in those edges for the switching. This explains for point \ref{cons5} of the construction.
\end{enumerate}
 Set $U =$ diag$\{U_{n1}, U_{n2} \dots U_{nk}, I_{|D|}\}$ . As each $U_i$ is unitary, $U$ is also unitary. Now $U G U = G^{(\pi)}$ which is the adjacency matrix of the graph $G^{(\pi)}$. Thus $G$ and $G^{\pi}$ are cospectral.
\end{proof}

\begin{Example}
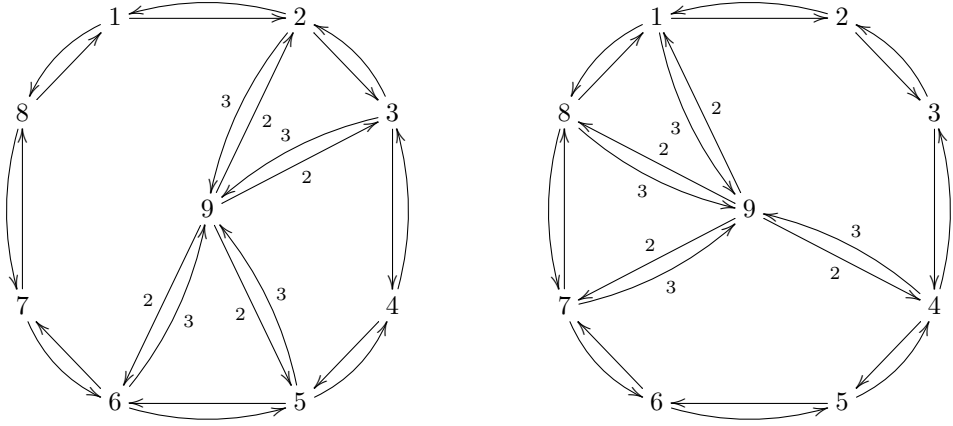
\begin{figure}
\xymatrix{  & 1 \ar[rr] \ar@/_/[ld]&
            & 2 \ar[rd] \ar@/_/[ll] \ar@/_/[ddl]_3&   & &
            & 1 \ar[rr] \ar@/_/[ld] \ar@/_/[ddr]_3&
            & 2 \ar[rd] \ar@/_/[ll]&  \\
           8 \ar[ur] \ar@/_/[dd]&  &  &  &
           3 \ar[dd] \ar@/_/[ul] \ar@/_/[dll]_3& &
           8 \ar[ur] \ar@/_/[dd] \ar@/_/[rrd]_3&  &  &  &
           3\ar[dd] \ar@/_/[ul]\\
            &  & 9 \ar[uur]_2 \ar[urr]_2 \ar[ddr]_2 \ar[ddl]_2&  &   & &
            &  & 9 \ar[uul]_2 \ar[rrd]_2 \ar[lld]_2 \ar[llu]_2&  &  \\
           7 \ar[uu] \ar@/_/[rd] &  &  &  &
           4 \ar[ld] \ar@/_/[uu] & &
           7 \ar[uu] \ar@/_/[rd] \ar@/_/[rru]_3&  &  &  &
           4 \ar[ld] \ar@/_/[uu] \ar@/_/[llu]_3\\
            & 6 \ar[ul] \ar@/_/[rr] \ar@/_/[uur]_3&
            & 5 \ar[ll] \ar@/_/[ur] \ar@/_/[uul]_3&   & &
            & 6 \ar[ul] \ar@/_/[rr]&
            & 5 \ar[ll] \ar@/_/[ur]&}
\caption{Graph Switching in Seidel Graph}
\label{switching}
\end{figure}
In the graph given in figure \ref{switching}, $C_1 = \{1 , 2 , 3 , 4, 5 , 6, 7, 8\}$ and $D = \{9\}$. For simplicity we have taken the subgraph induced by $C_1$ as an unweighed graph. Vertex $9$ was initially connected to the vertices $2, 3, 5$ and $6$. After transformation vertex $9$ is adjacent to the vertices $1, 7, 8$ and $4$. The graphs are not isomorphic but have same eigenvalues.
\end{Example}

\subsection{Construction of L-cospectral and Q-cospectral graphs}

We use Seidel switching on starlike graphs to generate L-cospectral and Q-cospectral graphs.

\begin{definition}
If a Seidel graph satisfies the following conditions, then the graph is called a Starlike graph.
\begin{enumerate}
\item
There is no edge from one vertex of $C_i$ to another vertex in $C_j$ for all $i\neq j$.
\item
Edges from Category 1 vertices of $D$ to the $C_i$ shall have equal weights, say $w_+$. Similarly, all the edges from the vertices of $C_i$ to the category 1 vertices in $D$ shall have equal weights say, $w_-$.
\item
Number of Category 2 vertices with respect to $C_i$ in $D$ is even. Other half are adjacent to $\frac{n_i}{2}$ number of vertices in $C_i$ and other half are adjacent to another $\frac{n_i}{2}$ number of vertices in $C_i$. Weights of the edges from category $2$ vertices to $C_i$ are the same, say $w^+$ and weights of the edges from vertices of $C_i$ to the category 2 vertices are also equal, say $w^-$.
\end{enumerate}
\end{definition}
{\bf Example:}
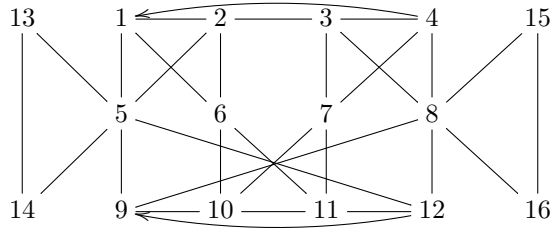
\begin{figure}
$$\xymatrix{13 \ar@{-}[dd] \ar@{-}[rd]&
           1 \ar@{-}[d] \ar@{-}[r] \ar@{-}[rd]&
           2 \ar@{-}[d] \ar@{-}[r] \ar@{-}[ld]&
           3 \ar@{-}[d] \ar@{-}[r] \ar@{-}[rd]&
           4 \ar@{-}[d] \ar@/_/[lll] \ar@{-}[ld]&
           15 \ar@{-}[dd] \ar@{-}[ld]\\
            & 5 \ar@{-}[d] \ar@{-}[drrr]&
              6 \ar@{-}[d] \ar@{-}[dr]&
              7 \ar@{-}[d] \ar@{-}[ld]&
              8 \ar@{-}[d] \ar@{-}[dlll]&   \\
                   14 \ar@{-}[ru]&
                   9 \ar@{-}[r]&
                   10 \ar@{-}[r]&
                   11 \ar@{-}[r]&
                   12 \ar@/^/[lll]&
                   16 \ar@{-}[lu]}$$
\caption{Starlike Graph}
\label{starlike}
\end{figure}
Figure~\ref{starlike} is an example of a starlike graph. Consider $C_1 = \{1, 2, 3, 4\}, C_2 = \{9, 10, 11, 12\}, C_3 = \{13, 14\}, C_4 = \{15, 16\}$ and $D = \{5, 6, 7, 8\}$. There is no connection between $C_i$ and $C_j$s for any $i$ and $j$.

Consider a starlike graph $G$. We follow the steps, given below, to construct $G'$ from $G$ using Seidel switching for weighted graphs described in the last subsection.
\begin{enumerate}
\item
For any graph $G$, Laplacian $L(G) = (l_{ij}(G))$ and signless Laplacian $Q(G) = (q_{ij}(G))$ matrices are defined by 
	$$l_{ij}(G) = \begin{cases} - a_{ij}(G) & \text{for $i \neq j$}\\ d_{i} - a_{ii}(G) & \text{for $i = j$}\end{cases}, ~\mbox{and}~ 
	q_{ij}(G) = \begin{cases} a_{ij}(G) & \text{for $i \neq j$}\\ d_{i} + a_{ii}(G) & \text{for $i = j$}\end{cases}.$$
	Construct a new graph $H$ with the adjacency matrix $A(H) = (a_{ij}(H))$ from $G$, such that, $a_{ij}(H) = l_{ij}(G)$, when we deal with the Laplacian matrix and $a_{ij}(H) = q_{ij}(G)$, when we deal with the signless Laplacian matrix.
\item
As $G$ is Seidel so is $H$. We do graph switching on $H$ and get a new graph $H^\pi$.
\item
To deal with Q-co-spectrality we construct a new graph $G'$, such that, $Q(G') = D(G') + A(G') =  A(H^\pi)$. 
$$a_{ij}(G') = \begin{cases} a_{ij}(H^\pi) & ~\text{for}~ i \neq j\\ \frac{1}{2} (a_{ii}(H^\pi) - \sum_{i \neq j} |a_{ij}(H^\pi)|) & ~\text{for}~ i \neq j \end{cases}.$$
To deal with L-co-spectrality we construct $L(G') = D(G') - A(G') =  A(H^\pi)$.
$$a_{ij}(G') = \begin{cases} -a_{ij}(H^\pi) & ~\text{when}~ i \neq j \\ 0 & ~\text{when}~ i = j \end{cases} .$$
\end{enumerate}

The changes we are doing on the graph $G$ can be represented diagrammatically by $$G \rightarrow H \hspace{.2cm} \underrightarrow{\text{Switching}} \hspace{.2 cm} H^\pi \rightarrow G'.$$

Hence, changes can be shown using Adjacency and Laplacian matrices as
 $$Q(G) (L(G)) = A(H) \hspace{.2cm} \underrightarrow{\text{Switching}} \hspace{.2cm} A(H^\pi) = Q(G') (L(G')). $$

\begin{Theorem}
Let $G$ be a Starlike graph and $G'$ be the graph constructed from $H$, which is obtained by using the signless Laplacian matrix $Q(G)$ (Laplacian matrix, L(G)) as mentioned above. Then $G$ and $G'$ are $Q$-co-spectral (L-co-spectral).
\end{Theorem}
\begin{proof}
We have $A(H^\pi) = UA(H)U$. As $Q(G) = Q(G')$, we have $Q(G') = UQ(G)U$. The matrix $U$ is symmetric and unitary. Thus, two graphs $G$ and $G'$ will have same signless Laplacian eigenvalues. Proof for L-co-spectrality is same.
\end{proof}

\begin{Example}
Consider the graphs in figure~\ref{dirnoniso}. Take $C_1 = \{9, 10\}, C_2 = \{5, 6, 7, 8\}$ and $D = \{1, 2, 3, 4\}$. Then the graphs satisfies all the conditions for being Starlike graphs. Both of them has same eigenvalues. Note that they are not isomorphic.
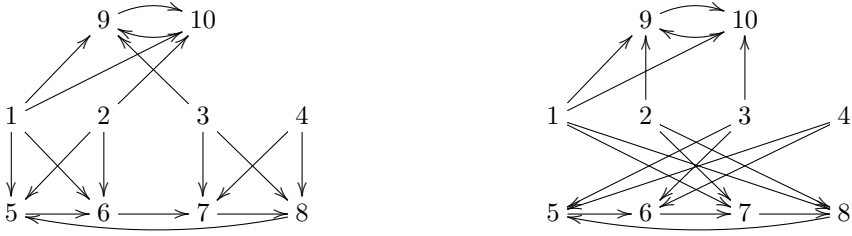
\begin{figure}
$$\xymatrix{ & 9 \ar@/^/[r] &
            10 \ar@/^/[l]&  & &&  &
            9  \ar@/^/[r]&
            10 \ar@/^/[l] &  & & \\
                        1 \ar@{->}[d] \ar@{->}[ru] \ar@{->}[urr] \ar@{->}[rd]&
                        2 \ar@{->}[d] \ar@{->}[ur] \ar@{->}[ld]&
                        3 \ar@{->}[d] \ar@{->}[ul] \ar@{->}[rd]&
                        4 \ar@{->}[d] \ar@{->}[ld]&
                        &&
                        1 \ar@{->}[rrd] \ar@{->}[rrrd] \ar@{->}[ru] \ar@{->}[urr]&
                        2 \ar@{->}[drr] \ar@{->}[u] \ar@{->}[rd]&
                        3 \ar@{->}[dll] \ar@{->}[u] \ar@{->}[dl]&
                        4 \ar@{->}[lld] \ar@{->}[llld]  \\
                                 5 \ar@{->}[r]&
                                 6 \ar@{->}[r]&
                                 7 \ar@{->}[r]&
                                 8 \ar@/^/[lll]&
                                 &&
                                 5 \ar@{->}[r]&
                                 6 \ar@{->}[r]&
                                 7 \ar@{->}[r]&
                                 8 \ar@/^/[lll]   \\}$$
\caption{Non-isomorphic graphs with same Laplacian eigenvalues}
\label{dirnoniso}
\end{figure}
\end{Example}

\begin{Theorem}
Loop weights of graphs $G$ and $G'$ are the same, that is, $a_{ii}(G) = a_{ii}(G')$.
\end{Theorem}
\begin{proof}
{\bf Case 1 :} The graph $G$ is a simple graph with or without loop.

Let there be $m_1$ number of non-loop edges incident on the vertex $i_k \in C_i$ in the subgraph induced by $C_i$. Let there be $s$ number of vertices of category 2 which are adjacent to $i_k$ and $t$ vertices of category 2 which are not connected to $i_k$. Then $q = s + t$.

Degree of $i_k$ at the vertex $G$ is
$$d_{i_k}(G) = \begin{cases}
p + m_1 + s + 1 & ~\text{if $i_k$ has a loop}\\
p + m_1 + s & ~\text{if $i_k$ has no loop}
\end{cases}.$$
So $i_ki_k$-th element of the Laplacian of the graph $G$ is
$$l_{i_ki_k}(G) = \begin{cases}
p + m_1 + s + 2 &\text{If $i_k$ has a loop}\\
p + m_1 + s &\text{If $i_k$ has no loop}
\end{cases}.$$
$H$ has a loop at the vertex $i_k$ of weight $l_{i_ki_k}(G)$, that is, $a_{i_ki_k}(G) = l_{i_ki_k}(G)$. As switching does not effect the loop weights, for graph $H^\pi$, $a_{i_ki_k}(H) = a_{i_ki_k}(H^\pi) = l_{i_ki_k}(G)$. Edges incident to the vertex $i_k$ at the graph $G'$ are
\begin{enumerate}
\item
$n$ number of edges which are inside the induced subgraph of $C_i$.
\item
$p$ number of edges from the vertices in $D$ of category 1.
\item
$t$ number of edges from the vertices of category 2 in $D$.
\item
A loop of weight $l$.
\end{enumerate}
After switching total weight of edges connected to $i_k$ is $m_1 + p + t + l$.

Now $Q(G') = A(H^\pi)$, thus $Q_{i_ki_k}(G') = A_{i_ki_k}(H^\pi)$. Thus
$$m_1 + p + t + 2l = \begin{cases}
p + m_1 + s + 2 &\text{if $i_k$ has a loop}\\
p + m_1 + s &\text{if $i_k$ has no loop}
\end{cases}$$
Simplifying we get
$$l = \begin{cases}
\frac{s - t + 2}{2} &\text{if there is a loop at the vertex $i_k$ of initial graph $H$}\\
\frac{s - t}{2} &\text{if there is no loop at the vertex $i_k$ of the initial graph $H$}
\end{cases}$$
We need loop weight $l \ge 0$. So either $s \ge t$ or $s \ge t - 2$ depending on existence or non-existence of loop at the vertex $i_k$. $s = t$ gives $l = 1$, if there is a loop at vertex $i_k$, and ($l = 0$) if there is no loop at the vertex $i_k$ of the graph $G$.

From this it follows that $D$ has an even number of vertices of Category 2 with respect to $C_i$. Half of them are connected to $\frac{n_i}{2}$ and the remaining half to the other $\frac{n_i}{2}$ vertices of $C_i$.

{\bf Case 2 :} $G$ in a weighted directed multi-graph.

Let order of all the vertices of $C_i$ be $m_2$. $a > 0$ is weight of the loop at the vertex $i_k$ of the graph $G$. Consider the weights of edges and loops incident to the vertex $i_k$ at $G$:
\begin{enumerate}
\item
$p$ edges from Category 1 vertices in $D$ to $i_k$ of weight $w_+$. $p$ edges from $i_k$ to Category 1 vertices in $D$ of weight $w_-$.
\item
$q$ edges from Category 2 vertices in $D$ to $i_k$ of weight $w^+$. $p$ edges from $i_k$ to Category 1 vertices in $D$ of weight $w^-$.
\item
Mod sum of all edges and loops in the subgraph induced by $C_i$ is $m_2$.
\end{enumerate}
\begin{equation*}
\begin{split}
\therefore  d_{i_k}(G) &= m_2 + p(|w_+| + |w_-|) + \frac{q}{2}(|w^+|+|w^-|)\\
a_{ii}(H^\pi) &= a_{ii}(H) = l_{ii}(G) = a + m_2 + p(|w_+| + |w_-|) + \frac{q}{2}(|w^+|+|w^-|)
\end{split}
\end{equation*}
Let $l =$ weight of loop at vertex $i$ in the graph $G'$.
\begin{equation*}
\begin{split}
& l_{ii}(G') = a_{ii}(H^\pi),\\
& m_2 - a + p(|w_+| + |w_-|) + \frac{q}{2}(|w^+|+|w^-|) + l + l =\\ 
& a + m_2 + p(|w_+| + |w_-|) + \frac{q}{2}(|w^+|+|w^-|),\\
& l = a.
\end{split}
\end{equation*}
Thus, in all the cases, the weight of the loop remains unchanged in the final graph.
\end{proof}

The conditions given here are all sufficient but not necessary. Consider the graphs in figure \ref{Nonisomorphic}. Here the two graphs are non-isomorphic but do not satisfy conditions for being starlike graphs. 
\begin{figure}
$$
\xymatrix{ & 9 \ar@/^/[r] &
            10 \ar@/^/[l]&  & &&  &
            9  \ar@/^/[r]&
            10 \ar@/^/[l] &  & & \\
                        1 \ar@{->}[d] \ar@{->}[ru] \ar@{->}[urr] &
                        2 \ar@{->}[d] \ar@{->}[u]&
                        3 \ar@{->}[d]&
                        4 \ar@{->}[d]&
                        &&
                        1 \ar@{->}[rd] \ar@{->}[ru] \ar@{->}[urr]&
                        2 \ar@{->}[drr] \ar@{->}[ur]&
                        3 \ar@{->}[dll]&
                        4 \ar@{->}[ld]   \\
                                 5 \ar@{->}[r]&
                                 6 \ar@{->}[r]&
                                 7 \ar@{->}[r]&
                                 8 \ar@/^/[lll]&
                                 &&
                                 5 \ar@{->}[r]&
                                 6 \ar@{->}[r]&
                                 7 \ar@{->}[r]&
                                 8 \ar@/^/[lll]   \\}$$
\caption{Non-isomorphic graphs with same Laplacian eigenvalues}
\label{Nonisomorphic}
\end{figure}
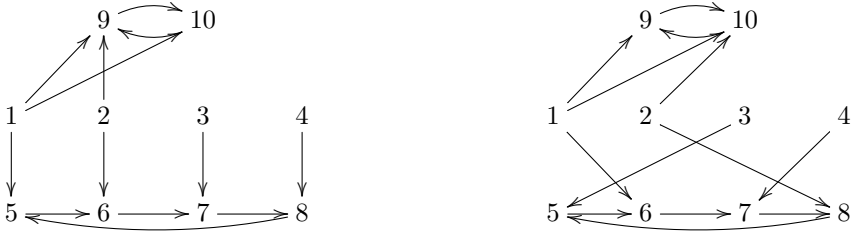

In this section, we have introduced star-like graphs for which the Seidel switching generates L and Q-co-spectral graphs. Let $H$ be L-cospectral graphs generated by Seidel switching from $G$. Quantum mechanically $\rho(H) = U^t\rho(G)U$. Here, we have dealt with two types of unitary operators $U_n = \frac{2}{n}J_n - I_n$ and $U = \operatorname{diag}\{U_{n_1}, U_{n_2}, \dots I\}$. In the next section, we shall describe quantum mechanical properties of these two operators. To the best of our knowledge, there has been no application of Seidel switching to quantum information processing reported in the literature.

\section{Quantum mechanical applications of Seidel switching}

Seidel switching deals with co-spectrality of graphs. Above, we used Seidel switching for generating Laplacian and signless Laplacian co-spectral graphs. Hence, it follows that the corresponding quantum state density matrices have the same spectra. Let $\rho$ be a density matrix with spectra $\Lambda(\rho) = \{\lambda_i: i = 1, 2, \dots n \}$. The von-Neumann entropy of $\rho$ is defined as, $S(\rho) = \tr(\rho \log(\rho)) = \sum_i \lambda_i \log(\lambda_i)$ \cite{nielsen2010quantum}. If two density matrices have equal eigenvalues, they have equal von-Neumann entropy. The proposed Seidel switching produces co-entropic weighted multi graphs. For simple graphs, coentropic graphs were identified in \cite{de2016interpreting}.

The density matrix of a pure state can be written as $\rho = \ket{\psi}\bra{\psi}$, corresponding to a state vector $\ket{\psi}$. Else it is a mixed state. In \cite{adhikari2012}, conditions for $\rho(G)$, corresponding to a weighted graph, to be pure or mixed were given.
\begin{Lemma}
The density matrix corresponding to a Laplacian or signless Laplacian matrix of a weighted digraph $G$ without loops has rank one, a pure state, if and only if the graph is $K_2$, that is a graph consists of two vertices and a connecting edge, or $K_2 \sqcup v_1 \sqcup v_2 \sqcup \dots v_{n−2}$, where, $v_1, v_2, \dots, v_{n−2}$ are isolated vertices.
\end{Lemma}
Let $G$ and $H$ be two star-like graphs s.t. $L(H) = UL(G)U^t$, where $U$ is a Seidel operator. In general, $G$ and $H$ are graphs with more than one loop or more than one edge. Hence, corresponding density matrices of the starlike graphs, $\rho(H)$ and $\rho(G)$ are mixed. Thus, $\rho(H) = U\rho(G)U^t$, shows that $U$ acts as a unitary evolution on mixed quantum states, when we apply Seidel switching on a star-like graph.

A quantum state in the Hilbert space $\mathcal{H}_1 \otimes \mathcal{H}_2$ is said to be separable if its density matrix, $\rho = \sum_i p_i \rho^{(1)}_i \otimes \rho^{(2)}_i$, where $\rho^{(1)}_i$ and $\rho^{(2)}_i$ represents quantum states belong to the individual Hilbert spaces $\mathcal{H}_1$ and $\mathcal{H}_2$. Otherwise, the state is entangled \cite{horodecki2009quantum}. Entanglement is used in different tasks of quantum information theory like quantum teleportation, coding and cryptography. 

The evolution from state $\ket{\phi}$ to another state $\ket{\psi}$ is determined by the equation $\ket{\psi} = U \ket{\phi}$. Here, $U$ is a unitary operator acting on the state vector $\ket{\phi}$. Under a global unitary transformation, some properties of $\ket{\phi}$ and $\ket{\psi}$, like entanglement, may differ. These changes are determined by the notion of the strength of the unitary operator $U$, a quantity which we will define and compute below. This strength can be measured from a number of perspectives. Crucial quantum information theoretic tasks depend on the proper choice of the unitary operator $U$. In this sense, quantum dynamics is considered as a measurable physical resource \cite{nielsen2003quantum}.

In quantum dynamics, entangled states are generated from separable states by a global unitary operation. Generally, a unitary operator $U$ acting on $\mathcal{H}^{(\otimes n)}$ is called a local unitary operator if $U = U_1 \otimes U_2 \otimes \dots \otimes U_n$, where, $U_i$ is a unitary operator acting on the individual Hilbert spaces $\mathcal{H}_i$. Otherwise, $U$ is a global unitary operator.

It is difficult to justify whether a unitary operator can be expressed as a tensor product of other unitary operators or not. A basic necessary condition for a local unitary operator is, the order of the matrix must be a prime number. Let the matrices, $U_i: i = 1, 2, \dots n$ be of order $m_i: i = 1, 2, \dots m_n$, respectively. Then their tensor product, $U_1 \otimes U_2 \otimes \dots U_n$ is of order $m_1m_2\dots m_n$. An operator of prime order can not act on state vectors of a Hilbert space $\mathcal{H}$ in the form $\mathcal{H}^{(\otimes n)}$. However, it should be noted that there are local unitary operators of composite order as well.

Very recently, a mathematical tool, matrix realignment, was used to resolve the problem of local or global unitary operations \cite{zhang2013criterion}. For a matrix $A = (a_{i,j})_{n \times n}$ we define a $1 \times n^2$ vector, $\vc(A) = (a_{1,1}, a_{1,2}, \dots a_{1,n}, a_{2,1}, \dots a_{n,n})$. Let, $U = (A_{ij})_{M\times M}$ be a block matrix with every block, $A_{ij}$, itself being a matrix of order $N \times N$. Then, realignment($U)$ is an $M^2 \times N^2$ matrix defined by,
$$\operatorname{Realingment}(U) = [\vc(A_{11}), \vc(A_{22}), \dots ,\vc(A_{MM})]^t.$$
We mention a lemma from \cite{zhang2013criterion}, which will be of use below.
\begin{Lemma}
An unitary matrix $U$ of order $MN$ can be represented as a tensor product of unitary matrices $u_1$ and $u_2$ of order $M$ and $N$ respectively, such that, $U = u_1 \otimes u_2$ if and only if rank(Realignment(U)) = 1.
\end{Lemma}

Seidel switching deals with two types of unitary operators, $U_n = \frac{n}{2}J_n - I_n$, and $U = U_{n_1} \oplus U_{n_2} \oplus \dots U_{n_k} \oplus I_{n_{k + 1}}$ of order $n$ and $n_1 + n_2 + \dots + n_k + n_{k + 1}$, respectively. Note that, a global unitary operator acts on a bipartite or a multipartite system. Dimension of any such system is always composite. Hence, here we consider only those Seidel operators whose order is a composite number. Trivially, any unitary operator with a prime order is always local. Also note that a composite number can be expressed an product of other numbers in many different ways. As an example we may write 12 as $2 \times 2 \times 3, 4 \times 3$ and $2 \times 6$. Thus, for a given composite number there may be many different Seidel operators according to these decompositions. We show that any Seidel operator with composite order is always a global unitary operator, by using the above lemma. 

\begin{Theorem}
For any composite number $n$, the Seidel operator, $U_n = \frac{n}{2}J_n - I_n$ is a global unitary operator.
\end{Theorem}

\begin{proof}
Let $U_n$ be a local unitary operator. For simplicity let $U_n = U_1 \otimes U_2$, where order of $U_1$ and $U_2$ are $p$ and $q$ respectively. Then, $n = p.q$. Moreover, $U_n$ can be partitioned into block matrices as follows.
$$U_n  = \begin{bmatrix}U_{11} & U_{12} & \dots U_{1p} \\ U_{21} & U_{22} & \dots U_{2p} \\ \vdots & \vdots & \hdots \vdots\\ U_{p1} & U_{p2} & \dots U_{pp} \end{bmatrix},$$
where, $U_{ii} = \frac{n}{2}J_p - I_p$ and $U_{ij} = \frac{n}{2} J_p$. These indicate,
\begin{align*}
\vc(U_{ii}) & = \left(\left(\frac{n}{2} - 1\right), \frac{n}{2}, \frac{n}{2}, \dots \frac{n}{2}, \frac{n}{2}, \left(\frac{n}{2} - 1\right), \dots , \left(\frac{n}{2} - 1\right)\right)^t,\\
\vc(U_{ij}) & = \left(\frac{n}{2}, \frac{n}{2}, \dots \frac{n}{2}\right)^t,\\
Realingment(U) & = [ \vc(U_{11}) , \vc(U_{22}) , \vdots , \vc(U_{pp})]^t.
\end{align*}
Note that, $\vc(U_{ii})$ and $\vc(U_{ij})$ are linearly independent. Hence, $Realignment(U)$ is of rank 2. Therefore, $U$ is not a local unitary operator. 
\end{proof}

\begin{Theorem}
The Seidel operator, $U = U_{n_1} \oplus U_{n_2} \oplus \dots U_{n_k} \oplus I_{n_{k + 1}}$, is a global unitary operator, provided $\sum_{i = 1}^{k + 1} n_i$ is a composite number.
\end{Theorem}

\begin{proof}
Note that, $I_{n_{k +1}} \neq U_{n_k}$ for any $U_{n_k}$. Hence, any block of $U$ containing $I_{n_{k +1}}$ will differ from any other block independent of partition on the matrix $U$. Thus, the rank of Realignment$(U)$ will always be more than 1. Therefore, $U$ will be a global unitary operator.
\end{proof}

In analogy with classical computation, logic gates are also used in quantum computation. Any unitary operator can be treated as a quantum logic gate. Pauli $X, Y, Z$ and Hadamard operator $H$ are familiar single qubit quantum gates. Graphical operation of some familiar quantum gates was studied extensively in \cite{dutta2016graph}. The next corollaries are interesting as they indicate links between quantum information and Seidel switching.
\begin{Corollary}
The Seidel operator, $U_2$, is a Pauli $X$ operator.
\end{Corollary}

\begin{proof}
$U_2 = J_2 - I_2 = \begin{bmatrix} 0 & 1 \\ 1 & 0\end{bmatrix} = X$.
\end{proof}

\begin{Corollary}
The Seidel operator $U = U_2 \oplus I_2$ is CNOT gate with the second qubit as control and first qubit as target.
\end{Corollary}

\begin{proof}
$U = U_2 \oplus I_2 = X \oplus I_2 = \begin{bmatrix} X & 0 \\ 0 & I_2\end{bmatrix}$.
\end{proof}

Next, we will compute the strength of Seidel operators in order to gauge their strength for global unitary operations. In particular we would like to quantitative estimates of the amount of entanglement that can be generated by such an operation. To do this, we use three different measures and make a comparison between them. The concept of operator Schmidt decomposition, introduced in \cite{miszczak2011singular} and which is connected to the singular value decomposition of the operator, plays an important role in these considerations. 

The set of all matrices of order $n$ over the complex number field is denoted by $M_n$. It forms a Hilbert space with an associated inner product, $\langle A, B \rangle = \tr(A^\dagger B)$. 
Here, dagger ($\dagger$) denotes conjugate transpose. This space is also called Hilbert-Schmidt space and denoted by $\mathcal{H}c^n$. Also, the inner product is called Hilbert-Schmidt inner product. An orthonormal standard basis for this space is
$$\mathcal{B}_{\mathcal{H}c^n} = \{E_{ij}; E_{ij} ~\mbox{is a matrix with all}~ 0 ~\mbox{but}~ 1 ~\mbox{at the}~ (i,j)\mbox{-th position}\}.$$
We enumerate this set as, $\mathcal{B}_{\mathcal{H}c^n} = \{E_i : i = 1, 2, \dots n^2\}$. Any matrix $U \in \mathcal{H}c^{mn}$, acting on the bipartite system $\mathcal{H}^m \otimes \mathcal{H}^n$ can be expressed as a linear combination in terms of the standard basis elements of $\mathcal{H}c^m$ and $\mathcal{H}c^n$ as follows,
$$U = \sum_{i = 1}^{m^2} \sum_{j = 1}^{n^2} c_{i,j} E_i E_j; c_{i,j} \in \mathbb{C}, E_i \in \mathcal{B}_{\mathcal{H}c^m}, E_j \in \mathcal{B}_{\mathcal{H}c^n}.$$
Singular values of the $m^2 \times n^2$ matrix $C = (c_{i,j})$ are the Schmidt coefficients of the operator $U$. We collect them as $\{s_i: i = 1, 2, \dots \min\{m^2, n^2\}\}$.

Note that, for any unitary operator $U$ of order $mn$, $\tr(U^\dagger U) = mn$. Also the ideas of singular value and Hilbert-Schmidt inner product indicates that $\sum_i s_i^2 = mn$. Thus, $\{\frac{s_i^2}{mn}\}$ generates a probability distribution. Hence, the non-locality of $U$ can be quantified by the Shanon entropy $H(.)$ of this distribution. This measure is called the Schmidt strength \cite{nielsen2003quantum} and is defined as,
$$K_{Sch}(U) = H\left(\left\{\frac{s_i^2}{mn}\right\}\right) = -\sum_i \frac{s_i^2}{mn} \log\left(\frac{s_i^2}{mn}\right).$$

Any unitary matrix of composite order has different sets of Schmidt coefficients as its order may be factored in a number of different ways. This leads to different values of Schmidt strengths for a given unitary operator. As an illustration, we have calculated the values of Schmidt strength of Seidel matrices of order till $100$. These are depicted in figure \ref{seidel_schmidt}, with the $x$-axis representing the order of the matrices and corresponding Schmidt values along the $y$ axis.
\begin{figure}
\begin{center}
\includegraphics[scale = .3]{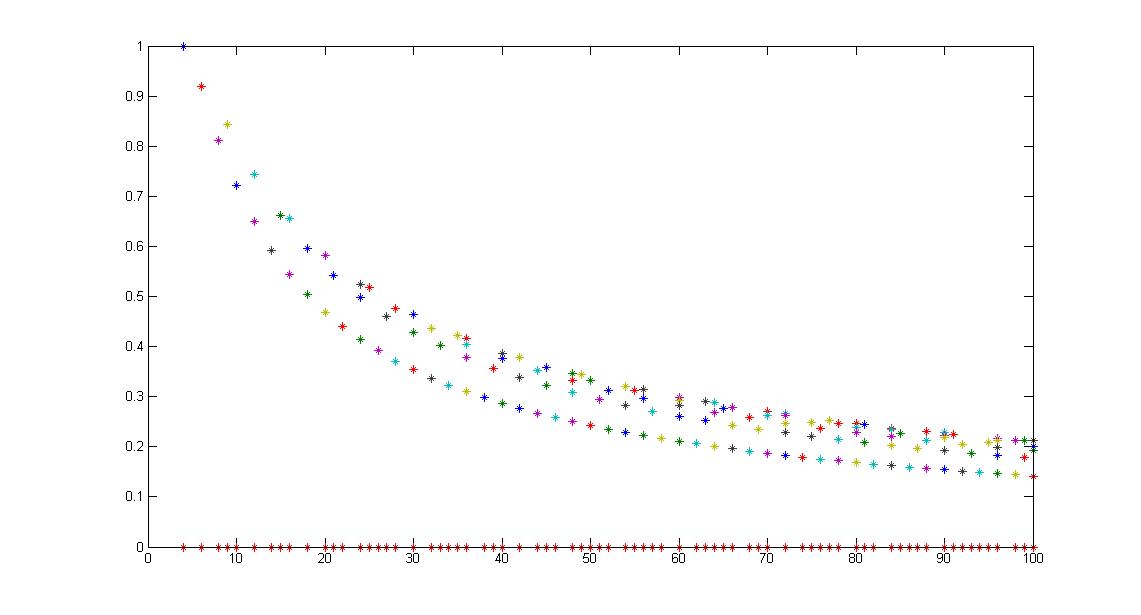}
\caption{Values of Schmidt strength of Seidel operator. Matrix order is plotted along the $x$-axis and the corresponding Schmidt values are plotted in the $y$-axis. Note that, there may be different Seidel operators with different powers for a given composite number. The CNOT gate, corresponding to an order four Seidel operator, has maximum power one.}
\label{seidel_schmidt}
\end{center}
\end{figure}
We can see from the figure \ref{seidel_schmidt} that the Schmidt strength decreases exponentially with the order of the Seidel matrix. It takes a maximum value of $1$ when the order is $4$, that is, for the CNOT gate.

 In \cite{wang2002quantum}, another measure for the strength of unitary operators $K_{WZ}$, in terms of entanglement generation, was provided. If the operator $U$ has the Schmidt coefficients $s_i$, then the strength can be expressed as,
$$K_{WZ} = 1 - \sum_i \frac{s_i^4}{m^2n^2}.$$
In figure \ref{seidel_zenerdi}, $K_{WZ}$ for all Seidel operators of order upto $100$ are plotted. A comparison of the measures $K_{Sch}$ and $K_{WZ}$, respectively, bring out that the entangling strength of the global unitary operations is maximum for the CNOT gate, represented by a matrix of order four and falls exponentially as the order increases. Also, it is seen that $K_{WZ}$ roughly scales as half of $K_{Sch}$. 

Another facet to understanding the strength of global unitary operations, vis-\'a-vis the entanglement generating capabilities from separable states would be to start with two different Hilbert spaces $\mathcal{H}_1$ and $\mathcal{H}_2$ with arbitrary ancillas, without any prior entanglement. Apply the global bipartite operator $U$ to generate entanglement. The strength of $U$ can be expressed as, $K_E(U) = \max_{\ket{\alpha} \ket{\beta}} E(U(\ket{\alpha} \otimes \ket{\beta})$. Here, $\ket{\alpha}$ and $\ket{\beta}$ runs over all pure state on $\mathcal{H}_1$ and $\mathcal{H}_2$ with ancillas $\mathcal{R}_{\mathcal{H}_1}$ and $\mathcal{R}_{\mathcal{H}_2}$, respectively. Here, $E$ is the usual measure of entanglement, that is, the von-Neumann entropy of the reduced density matrix. It is proved in \cite{nielsen2003quantum}, that $K_{Sch}$ acts as a lower bound of $K_E(U)$. We have seen that $K_{Sch}$ decreases exponentially with the order of the Seidel operator. From this, it could be conjectured that generation of entanglement, by the application of the Seidel operation, decreases exponentially with the order of the Seidel matrix. This idea is supported by the CNOT gate (Seidel operator of order $4$) which generates maximally entangled states from separable states and has the maximum value of $K_{Sch}$. 

\begin{figure}
\begin{center}
\includegraphics[scale=.25]{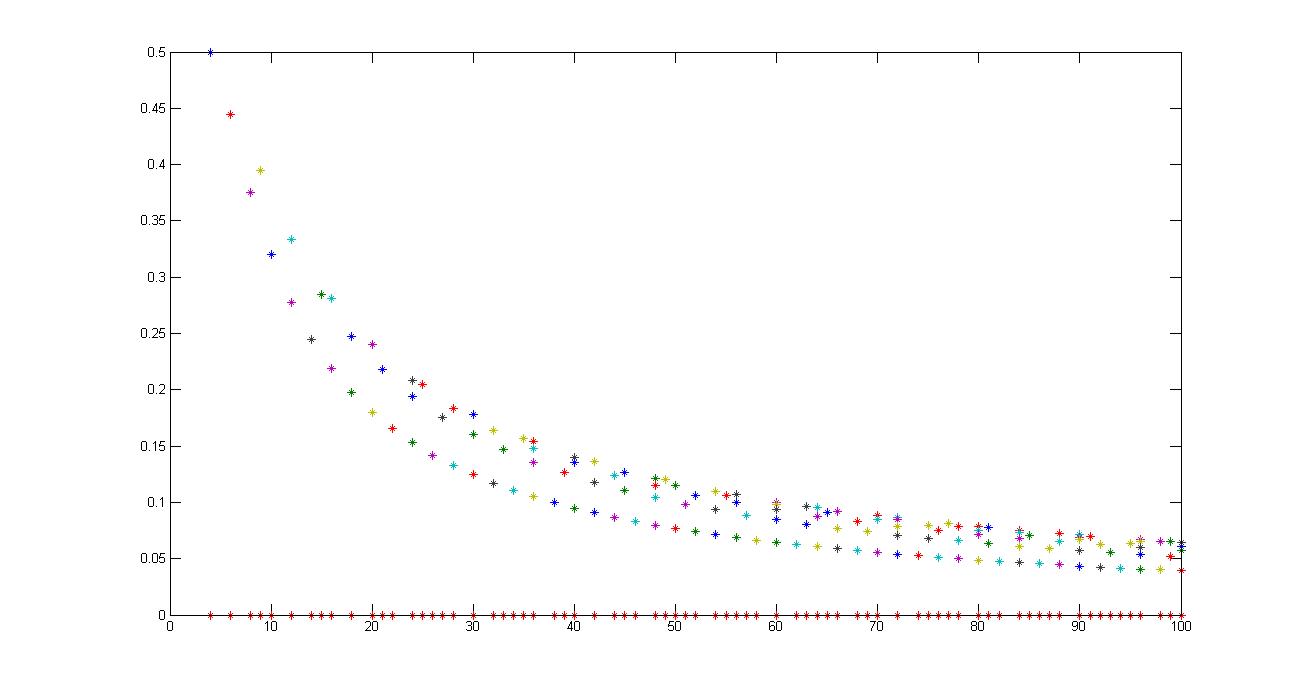}
\caption{plot of $K_{WZ}$ w.r.t. the order of Seidel operators up to order 100. Orders of the matrices are plotted along the x axis and the strength of the corresponding operator along the y axis. For order 4, CNOT gate has maximum strength.}
\label{seidel_zenerdi}
\end{center}
\end{figure}

\section{Conclusion}

Inspired by the concept of Seidel graph switching for simple graphs, we developed switching for Seidel graphs that can produce cospectral weighted digraphs with multiple edges and loops. To the best of our knowledge, this is the first instance of a use being made of Seidel switching to quantum information. Here Seidel switching is depicted as a unitary operation, useful for generating cospectral graphs. This brings to light a non-trivial side of Seidel witching, that is, a global quantum operation on graphs.

In the context of relating density matrices corresponding to a graph to their Laplacian and signless Laplacian matrices, we have applied Seidel switching to construct Laplacian and signless Laplacian cospectral graphs. Hence, Seidel switching helps us to generate density matrices with equal spectra. The corresponding quantum states have equal von Neumann entropy. We also discuss, quantum mechanical properties of unitary matrices which are closely related to Seidel switching, named Seidel operator. Interesting examples of such operators are Pauli X as well as the CNOT gate. We have shown that every Seidel operators of composite order is a global unitary operator. We have computed their entanglement generating strength.

This work elucidates a link between a well known mathematical technique, that is, Seidel switching and quantum information. This work will hopefully lead to attempts on the following problems:
\begin{enumerate}
	\item
		All quantum states related to the Starlike graphs are mixed states. Seidel switching is capable of generating entangled quantum states. How much maximal entanglement may be extracted from them? This example of entanglement distillation would justify the role of Seidel switching to a number of  quantum information tasks.
	\item
		Seidel switching will play a central role in problems related to co-spectrality of quantum states in the context of quantum information.
\end{enumerate}

\section*{Acknowledgement}

This work is partly supported by the project {\it Graph Theoretical Aspects of Quantum Information Processing} (Project No. 25(0210)/13/EMR-II) funded by the Council of Science and Industrial Research, New Delhi. SD is thankful for doctoral fellowship to the Ministry of Human Recourse Development, Government of India.

\bibliographystyle{vancouver}

\begin{thebibliography}{10}
	
	\bibitem{braunstein2006laplacian}
	Braunstein SL, Ghosh S, Severini S.
	\newblock The Laplacian of a graph as a density matrix: a basic combinatorial
	approach to separability of mixed states.
	\newblock Annals of Combinatorics. 2006;10(3):291--317.
	
	\bibitem{adhikari2012}
	Adhikari B, Adhikari S, Banerjee S, Kumar A.
	\newblock Laplacian matrices of weighted diagraph and graph representation of
	quantum states.
	\newblock sumbitted. 2016;.
	
	\bibitem{west2001introduction}
	West DB, et~al.
	\newblock Introduction to graph theory. vol.~2.
	\newblock Prentice hall Upper Saddle River; 2001.
	
	\bibitem{bapat2010graphs}
	Bapat RB.
	\newblock Graphs and matrices.
	\newblock Springer; 2010.
	
	\bibitem{lovasz1979shannon}
	Lov{\'a}sz L.
	\newblock On the Shannon capacity of a graph.
	\newblock IEEE Transactions on Information theory. 1979;25(1):1--7.
	
	\bibitem{nielsen2010quantum}
	Nielsen MA, Chuang IL.
	\newblock Quantum computation and quantum information.
	\newblock Cambridge university press; 2010.
	
	\bibitem{berkolaiko2013introduction}
	Berkolaiko G, Kuchment P.
	\newblock Introduction to quantum graphs.
	\newblock 186. American Mathematical Soc.; 2013.
	
	\bibitem{cabello2014}
	Cabello A, Severini S, Winter A.
	\newblock Graph-theoretic approach to quantum correlations.
	\newblock Physical review letters. 2014;112(4):040401.
	
	\bibitem{merris1994laplacian}
	Merris R.
	\newblock Laplacian matrices of graphs: a survey.
	\newblock Linear algebra and its applications. 1994;197:143--176.
	
	\bibitem{cvetkovic2009towards1}
	Cvetkovi{\'c} D, Simi{\'c} SK.
	\newblock Towards a spectral theory of graphs based on the signless Laplacian,
	I.
	\newblock Publ Inst Math(Beograd). 2009;85(99):19--33.
	
	\bibitem{horodecki2009quantum}
	Horodecki R, Horodecki P, Horodecki M, Horodecki K.
	\newblock Quantum entanglement.
	\newblock Reviews of modern physics. 2009;81(2):865.
	
	\bibitem{dutta2016bipartite}
	Dutta S, Adhikari B, Banerjee S, Srikanth R.
	\newblock Bipartite separability and nonlocal quantum operations on graphs.
	\newblock Phys Rev A. 2016 Jul;94:012306.
	
	\bibitem{ozolsgeneralized}
	Ozols M, Mancinska L.
	\newblock Generalized Bloch Vector and the Eigenvalues of a Density
	Matrix;\url{http://home.lu.lv/~sd20008/papers/Bloch%20Vectors%20and%20Eigenvalues.pdf}.
		
		\bibitem{Seidel1974}
		Seidel JJ.
		\newblock Graphs and two-graphs.
		\newblock In: Proceedings of the {F}ifth {S}outheastern {C}onference on
		{C}ombinatorics, {G}raph {T}heory and {C}omputing ({F}lorida {A}tlantic
		{U}niv., {B}oca {R}aton, {F}la., 1974). Utilitas Math., Winnipeg, Man.; 1974.
		p. 125--143. Congressus Numerantium, No. X.
		
		\bibitem{butler2010note}
		Butler S.
		\newblock A note about cospectral graphs for the adjacency and normalized
		Laplacian matrices.
		\newblock Linear and Multilinear Algebra. 2010;58(3):387--390.
		
		\bibitem{emms2009coined}
		Emms D, Severini S, Wilson RC, Hancock ER.
		\newblock Coined quantum walks lift the cospectrality of graphs and trees.
		\newblock Pattern Recognition. 2009;42(9):1988--2002.
		
		\bibitem{halbeisen2000reconstruction}
		Halbeisen L, Hungerb{\"u}hler N.
		\newblock Reconstruction of weighted graphs by their spectrum.
		\newblock European Journal of Combinatorics. 2000;21(5):641--650.
		
		\bibitem{milanese2010approximating}
		Milanese A, Sun J, Nishikawa T.
		\newblock Approximating spectral impact of structural perturbations in large
		networks.
		\newblock Physical Review E. 2010;81(4):046112.
		
		\bibitem{godsil1982}
		Godsil CD, McKay BD.
		\newblock Constructing cospectral graphs.
		\newblock Aequationes Math. 1982;25(2-3):257--268.
		
		\bibitem{de2016interpreting}
		de~Beaudrap N, Giovannetti V, Severini S, Wilson R.
		\newblock Interpreting the von Neumann entropy of graph Laplacians, and
		coentropic graphs.
		\newblock A Panorama of Mathematics: Pure and Applied. 2016;658:227.
		
		\bibitem{nielsen2003quantum}
		Nielsen MA, Dawson CM, Dodd JL, Gilchrist A, Mortimer D, Osborne TJ, et~al.
		\newblock Quantum dynamics as a physical resource.
		\newblock Physical Review A. 2003;67(5):052301.
		
		\bibitem{zhang2013criterion}
		Zhang TG, Zhao MJ, Li M, Fei SM, Li-Jost X.
		\newblock Criterion of local unitary equivalence for multipartite states.
		\newblock Physical Review A. 2013;88(4):042304.
		
		\bibitem{dutta2016graph}
		Dutta S, Adhikari B, Banerjee S.
		\newblock A graph theoretical approach to states and unitary operations.
		\newblock Quantum Information Processing. 2016;15(5):2193--2212.
		
		\bibitem{miszczak2011singular}
		Miszczak JA.
		\newblock Singular value decomposition and matrix reorderings in quantum
		information theory.
		\newblock International Journal of Modern Physics C. 2011;22(09):897--918.
		
		\bibitem{wang2002quantum}
		Wang X, Zanardi P.
		\newblock Quantum entanglement of unitary operators on bipartite systems.
		\newblock Physical Review A. 2002;66(4):044303.
		
	\end{thebibliography}

\end{document}